\documentclass{article}
\usepackage{amsfonts}
\usepackage{amsmath}
\usepackage{amssymb}
\usepackage{amsthm}
\usepackage[margin=1.3in]{geometry}
\usepackage{hyperref}

\newtheorem{thm}{Theorem}
\newtheorem{lemma}{Lemma}

\newtheorem{corollary}{Corollary}

\newcommand{\red}{\text{red}}
\newcommand{\weight}{\text{weight}}
\newcommand{\inv}{\text{inv}}
\newcommand{\omat}{\boldsymbol{1}_{n}}
\newcommand{\zmat}{\boldsymbol{0}_{n}}

\begin{document}

\title{Using functional equations to enumerate $1324$-avoiding permutations}

\author{Fredrik Johansson\thanks{RISC, Johannes Kepler University, 4040 Linz, Austria. [fredrik.johansson@risc.jku.at]} \; and Brian Nakamura\thanks{CCICADA, Rutgers University-New Brunswick, Piscataway, NJ, USA. [bnaka@dimacs.rutgers.edu]}}

\date{}

\maketitle

\begin{abstract}
	We consider the problem of enumerating permutations with exactly $r$ occurrences of the pattern $1 3 2 4$ and derive functional equations for this general case as well as for the pattern avoidance ($r=0$) case. The functional equations lead to a new algorithm for enumerating length $n$ permutations that avoid $1 3 2 4$. This approach is used to enumerate the $1 3 2 4$-avoiders up to $n = 31$. We also extend those functional equations to account for the number of inversions and derive analogous algorithms.
\end{abstract}

\section{Introduction}

Let $a_{1} \ldots a_{k}$ be a sequence of $k$ distinct positive integers. We define the \emph{reduction} of this sequence, denoted by $\red(a_{1} \ldots a_{k})$, to be the length $k$ permutation $\tau = \tau_{1} \ldots \tau_{k}$ that is order-isomorphic to $a_{1} \ldots a_{k}$ (i.e., $a_{i} < a_{j}$ if and only if $\tau_{i} < \tau_{j}$ for every $i$ and $j$). Given a (permutation) pattern $\tau \in \mathcal{S}_{k}$, we say that a permutation $\pi = \pi_{1} \ldots \pi_{n}$ \emph{contains} the pattern $\tau$ if there exists $1 \leq i_{i} < i_{2} < \ldots < i_{k} \leq n$ such that $\red(\pi_{i_{1}} \pi_{i_{2}} \ldots \pi_{i_{k}}) = \tau$, in which case we call $\pi_{i_{1}} \pi_{i_{2}} \ldots \pi_{i_{k}}$ an \emph{occurrence} of $\tau$. We also define $N_{\tau}(\pi)$ to be the number of occurrences of pattern $\tau$ in the permutation $\pi$. For example, if the pattern $\tau = 1 2 3$, the permutation $5 3 4 1 2$ avoids the pattern $\tau$ (so $N_{1 2 3}(5 3 4 1 2) = 0$), whereas the permutation $5 2 1 3 4$ contains two occurrences of $\tau$ (so $N_{1 2 3}(5 2 1 3 4) = 2$).

For a pattern $\tau$ and non-negative integer $r \geq 0$, we define the set
\begin{align*}
	\mathcal{S}_{n}(\tau,r) := \left\{ \pi \in \mathcal{S}_{n} : \pi \text{ has exactly } r \text{ occurrences of the pattern } \tau \right\}
\end{align*}
and also define $s_{n}(\tau,r) := \left| \mathcal{S}_{n}(\tau,r) \right|$. For the $r=0$ case, we say that the two patterns $\sigma$ and $\tau$ are \emph{Wilf-equivalent} if $s_{n}(\sigma, 0) = s_{n}(\tau,0)$ for all $n$. Additionally, two patterns that are Wilf-equivalent are said to belong to the same \emph{Wilf-equivalence class}. Note that the $r=0$ case corresponds with ``classical'' pattern avoidance, which has been well-studied. Work on the more general problem ($r \geq 0$) has usually been restricted to patterns of length $3$ and small $r$. Some recent work include \cite{bona:f132, bona:prec132, bur:f321, callan, fulmek, manvain, noonan, noonzeil:forbid}.

A little more is known for the pattern avoidance problem, but the problem quickly gets difficult. For each $\tau \in \mathcal{S}_{3}$, it is well known that $s_{n}(\tau,0) = \frac{1}{n+1}{2 n \choose n}$ (the Catalan numbers) \cite{knuth}. For the length $4$ patterns, there are three cases (Wilf-equivalence classes) to consider: $1 2 3 4$, $1 3 4 2$, and $1 3 2 4$. The enumeration for the $1 2 3 4$-avoiding permutations was solved by Gessel in \cite{gessel}. Later, B{\'o}na solved the case for the $1 3 4 2$-avoiding permutations in \cite{bona:wc1342}. The pattern $1 3 2 4$, however, has been notoriously difficult to enumerate.

There is currently no non-recursive formula known for computing $s_{n}(1 3 2 4, 0)$, and precise asymptotics are not known either. Marinov and Radoi{\v{c}}i{\'c} developed an approach in \cite{marrad} using generating trees and computed $s_{n}(1 3 2 4, 0)$ for $n \leq 20$. Another approach (using insertion encoding) was used by Albert et al.~in \cite{aerwz} to compute $s_{n}(1 3 2 4, 0)$ for $n \leq 25$ (five more terms). Given the difficulty of this pattern, Zeilberger has even conjectured that ``Not even God knows $s_{1000}(1 3 2 4, 0)$'' \cite{eldvat}.

Given the difficulty of exact enumeration, work has also been done on studying how the sequence $s_{n}(\tau,0)$ grows for various patterns. We define the \emph{Stanley-Wilf limit} of a pattern $\tau$ to be:
\begin{align}
	L(\tau) := \mathop{\lim} \limits_{n \rightarrow \infty} {(s_{n}(\tau, 0))^{1/n}}.
\end{align}
Thanks to results by Arratia \cite{arratia} and Marcus and Tardos \cite{martar}, we know that for each pattern $\tau$, the limit $L(\tau)$ exists and is finite. For patterns of length three, the Stanley-Wilf limit is known to be $4$. Additionally, Regev \cite{regev} showed that $L(1 2 3 4) = 9$, while B{\'o}na's result in \cite{bona:wc1342} gives us $L(1 3 4 2) = 8$. The exact limit for the pattern $1 3 2 4$, however, is still unknown. The best known lower bound is by Albert et al. \cite{aerwz}, who showed that $L(1 3 2 4) \geq 9.47$. The best known upper bound has seen some improvements in recent years. The recent ``best'' upper bound was improved by Claesson, Jel{\'{\i}}nek, and Steingr{\'{\i}}msson in \cite{cjs:ub1324} to $L(1 3 2 4) \leq 16$. That approach was then refined by B{\'o}na in \cite{bona:ub1324} to show that $L(1 3 2 4) < (7 + 4 \sqrt{3}) \approx 13.93$.

Additionally, Claesson, Jel{\'{\i}}nek, and Steingr{\'{\i}}msson conjectured that the number of length $n$ permutations avoiding $1 3 2 4$ with exactly $k$ inversions was non-decreasing in $n$ (for each fixed $k$). They show that if this conjecture holds, then $L(1 3 2 4) \leq e^{\pi \sqrt{2/3}} \approx 13.002$. Neither the current lower bound nor this potential new upper bound appear to be ``close'' to the exact value of $L(1 3 2 4)$. For example, Steingr{\'{\i}}msson's survey paper \cite{stein:survey} considers empirical data and suggests that it may be close to $11$. Some data we consider in this paper suggests a similar story.

The paper is organized in the following manner. In Section~\ref{pat1324}, we derive functional equations for computing $s_{n}(1 3 2 4,r)$. Furthermore, the approach is specialized to the avoidance case ($r=0$) to derive an algorithm for enumerating the $1 3 2 4$-avoiding permutations. In Section~\ref{compres}, we describe technical details on the algorithm and use it to compute $s_{n}(1 3 2 4,0)$ up to $n=31$, giving us $6$ new terms. We use this new data to make some empirical observations. In Section~\ref{permstats}, we extend the functional equations to keep track of the number of inversions. We conclude with a few final remarks in Section~\ref{concl}. \\

\section{Functional equations for the pattern $1 3 2 4$}\label{pat1324}

We begin by extending the functional equations approach in \cite{naka:gwilf2} to the pattern $1 3 2 4$. We will first derive functional equations that can be used to compute $s_{n}(1 3 2 4, r)$. The approach will be presented in full detail so that this article is self-contained. We then specialize this approach to the $r=0$ case and develop a new algorithm for enumerating the $1 3 2 4$-avoiding permutations.

\subsection{A general approach to $s_{n}(1 3 2 4, r)$}

Given a non-negative integer $n$, we define the polynomial (in the variable $t$)
\begin{align}
	f_{n}(t) := \mathop{\sum} \limits_{\pi \in \mathcal{S}_{n}} t^{N_{1 3 2 4}(\pi)}.
\end{align}
Observe that the coefficient of $t^{r}$ in $f_{n}(t)$ is exactly equal to $s_{n}(1 3 2 4, r)$.

In addition to the variable $t$, we introduce $n(n+1)/2$ catalytic variables $x_{i,j}$ with $1 \leq i \leq j \leq n$ and $n(n+1)/2$ catalytic variables $y_{i,j}$ with $1 \leq j \leq i \leq n$. Note that the subscripts of the two sets of catalytic variables range over different quantities. We define the weight of a permutation $\pi = \pi_{1} \ldots \pi_{n}$ to be 
\begin{gather*}
	\weight(\pi) := \\
	t^{N_{1 3 2 4}(\pi)} \mathop{\prod} \limits _{1 \leq i \leq j \leq n} x_{i,j}^{\# \{ (a,b) \; : \; \pi_{a} < \pi_{b} , \; \pi_{a} = i  , \; \pi_{b} > j\}} \cdot \mathop{\prod} \limits _{1 \leq j \leq i \leq n} y_{i,j}^{\# \{ (a,b,c) \; : \; \pi_{b} < \pi_{a} < \pi_{c} , \; \pi_{a} = i  ,\; \pi_{b} \geq j\}}
\end{gather*}
where it is always assumed that $1 \leq a < b < c \leq n$. For example, 
\begin{align*}
	\weight(2 1 3) &= x_{1,1} x_{1,2} x_{2,1} x_{2,2} \cdot y_{2,1}\\
	\weight(4 1 3 2 5) &= t \cdot x_{1,1}^{3} x_{1,2}^{2} x_{1,3} x_{1,4} x_{2,2} x_{2,3} x_{2,4} x_{3,3} x_{3,4} x_{4,4} \cdot y_{3,1} y_{3,2} y_{4,1}^{3} y_{4,2}^{2} y_{4,3}.
\end{align*}
In essence, the weight stores information about $1 3 2 4$ patterns, $2 1 3$ patterns (which may become the ``$3 2 4$'' of a $1 3 2 4$ occurrence), and $1 2$ patterns (which may become the ``$1 3$'' of a $2 1 3$ occurrence) through the exponents of the variable $t$, the variables $y_{i,j}$, and the variables $x_{i,j}$, respectively.

We will define a multi-variate polynomial $P_{n}$ on all the previously defined variables. For notational convenience, we first write the $x_{i,j}$ variables and the $y_{i,j}$ variables as matrices of variables:
\begin{align}
	X_{n} := \left[ \begin{array}{ccccc} 
		x_{1,1} & & \cdots & & x_{1,n}\\
		 & \ddots & & & \\
		\vdots & & x_{i,i} & & \vdots \\
		 & & & \ddots & \\
		x_{n,1} & & \cdots & & x_{n,n}
		\end{array} \right], \; \; \; \;
	Y_{n} := \left[ \begin{array}{ccccc} 
		y_{1,1} & & \cdots & & y_{1,n}\\
		 & \ddots & & & \\
		\vdots & & y_{i,i} & & \vdots \\
		 & & & \ddots & \\
		y_{n,1} & & \cdots & & y_{n,n}
		\end{array} \right]
\end{align}
where we will disregard the entries below the diagonal in $X_{n}$ and the entries above the diagonal in $Y_{n}$.

For each $n$, we now define the multi-variate polynomial:
\begin{align*}
	P_{n}(t; X_{n}, Y_{n}) := \mathop{\sum} \limits_{\pi \in \mathcal{S}_{n}} \weight(\pi)
\end{align*}
Observe that $P_{n}(t; \omat, \omat) = f_{n}(t)$ is our desired polynomial, where $\omat$ is the $n \times n$ matrix of all $1$'s. For a fixed $r \geq 0$, our goal is to quickly compute the coefficient of $t^{r}$ in $P_{n}(t; \omat, \omat)$, which is exactly $s_{n}(1 3 2 4,r)$. We will do this by deriving a functional equation for $P_{n}$. This follows readily from the following result:

\begin{lemma}
	\label{lemma1324}
	Let $\pi = \pi_{1} \ldots \pi_{n}$ and suppose that $\pi_{1} = i$. If $\pi^{\prime} := \red(\pi_{2} \ldots \pi_{n})$, then
	\begin{align*}
		\weight(\pi) = x_{i,i}^{n-i} x_{i,i+1}^{n-i-1} \ldots x_{i,n-1}^{1} \cdot \weight(\pi^{\prime}) \vert_{A^{\prime}} \quad ,
	\end{align*}
	where $A^{\prime}$ is the set of substitutions given by
	\begin{align*}
		A^{\prime} := \begin{cases}
			x_{b,c} \rightarrow x_{b,c+1} & b < i, c \geq i\\
			x_{b,c} \rightarrow x_{b+1,c+1} & b \geq i, c \geq i\\
			x_{b,c} \rightarrow y_{i,1} y_{i,2} \ldots y_{i,b} \cdot x_{b,c} \cdot x_{b,c+1} & b < i, c = i-1\\
			y_{b,c} \rightarrow y_{b+1,c} & b \geq i, c < i\\
			y_{b,c} \rightarrow y_{b+1,c+1} & b \geq i, c > i\\
			y_{b,c} \rightarrow t y_{b+1,c} \cdot y_{b+1,c+1} & b \geq i, c = i \quad .
		\end{cases}
	\end{align*}
\end{lemma}

\begin{proof}
	We assume $i$ to be a fixed value. Observe that $N_{1 3 2 4}(\pi)$ is equal to the number of occurrences of $1 3 2 4$ in $\pi_{2} \ldots \pi_{n}$ \textbf{plus} the number of occurrences of $2 1 3$ in $\pi_{2} \ldots \pi_{n}$, where the term corresponding to the ``$1$'' is greater than $i$.
	
	If we re-insert $i$ at the beginning of $\pi^{\prime}$, we would shift all the terms $i, i+1, \ldots , n-1$ up by $1$. This gives us the substitutions:
	\begin{align*}
			x_{b,c} & \rightarrow x_{b,c+1} && b < i, c \geq i\\
			x_{b,c} & \rightarrow x_{b+1,c+1} && b \geq i, c \geq i\\
			y_{b,c} & \rightarrow y_{b+1,c} && b \geq i, c < i\\
			y_{b,c} & \rightarrow y_{b+1,c+1} && b \geq i, c > i.
	\end{align*}
	
	We make a few more observations. First, in $\weight(\pi)$, the exponents of $x_{k,i-1}$ and $x_{k,i}$ are equal and the exponents of $y_{k,i}$ and $y_{k,i+1}$ are equal for each $k$ (since $\pi_{1} = i$). This gives the substitutions $x_{b,i-1} \rightarrow x_{b,i-1} \cdot x_{b,i}$ if $b < i$ and $y_{b,i} \rightarrow y_{b+1,i} \cdot y_{b+1,i+1}$ if $b \geq i$.
	
	Second, the number of $1 3 2 4$ patterns that include the first term $\pi_{1} = i$ is the sum of the exponents of $y_{j,i+1}$ for $i+1 \leq j \leq n$. The substitution $y_{b,i} \rightarrow y_{b+1,i} \cdot y_{b+1,i+1}$ changes to $y_{b,i} \rightarrow t y_{b+1,i} \cdot y_{b+1,i+1}$ (for $b \geq i$).
	
	Third, the number of $2 1 3$ patterns that include the first term $\pi_{1} = i$ (i.e., the ``$2$'' term is equal to $i$) and whose ``$1$'' term is at least $k$ is equal to the sum of the exponent of $x_{j,i}$ for $k \leq j \leq i-1$. The substitution $x_{b,i-1} \rightarrow x_{b,i-1} \cdot x_{b,i}$ changes to $x_{b,i-1} \rightarrow y_{i,1} y_{i,2} \ldots y_{i,b} \cdot x_{b,i-1} \cdot x_{b,i}$ (for $b < i$).
	
	This gives us our substitution set $A^{\prime}$. Finally, the new ``$i$'' would create new $1 2$ patterns and would require an extra factor of $x_{i,i}^{n-i} x_{i,i+1}^{n-i-1} \ldots x_{i,n-1}^{1}$ for the weight. 
\end{proof}

Now, we define the operator $R_{1}$ on an arbitrary $n \times n$ square matrix $Y_{n}$ and $i < n$ to be:
\begin{align}
	R_{1}(Y_{n},i) := \left[ \begin{array}{ccccccc} 
		y_{1,1} & \cdots & y_{1,i-1} & t y_{1,i} y_{1,i+1} & y_{1,i+2} & \cdots & y_{1,n} \\
		\vdots & \ddots & & \vdots & & & \vdots \\
		y_{i-1,1} &  & y_{i-1,i-1} &  & \cdots & & y_{i-1,n} \\
		y_{i+1,1} & \cdots & y_{i+1,i-1} & t y_{i+1,i} y_{i+1,i+1} & y_{i+1,i+2} & \cdots & y_{i+1,n} \\
		\vdots & & \vdots & \vdots & \ddots & & \vdots\\
		\vdots & & \vdots & \vdots & & \ddots & \vdots\\
		y_{n,1} & \cdots & y_{n,i-1} & t y_{n,i} y_{n,i+1} & y_{n,i+2} & \cdots & y_{n,n}
		\end{array} \right]. \label{R1}
\end{align}
If $i = n$, then $R_{1}(Y_{n},i)$ is defined to be the $(n-1) \times (n-1)$ matrix obtained by deleting the $n$-th row and $n$-th column from $Y_{n}$. In essence, the $R_{1}$ operator deletes the $i$-th row, merges the $i$-th and $(i+1)$-th columns (via term-by-term multiplication), and multiplies this new column by a factor of $t$. It is important to note that while this operator is defined on any $n \times n$ matrix, it will only be applied to our ``matrix of variables'' $Y_{n}$ to get a smaller $(n-1) \times (n-1)$ matrix.

In addition, we define another operator $R_{2}$ on two $n \times n$ square matrices $X_{n}$ and $Y_{n}$ and $1 < i \leq n$ to be: 
\begin{align}
	R_{2}(X_{n},Y_{n},i) := \left[ \begin{array}{ccccccc} 
		x_{1,1} & \cdots & x_{1,i-2} & w_{1} & x_{1,i+1} & \cdots & x_{1,n} \\
		\vdots & \ddots & & \vdots & & & \vdots \\
		x_{i-2,1} & \cdots & x_{i-2,i-2} & w_{i-2} & x_{i-2,i+1} & \cdots & x_{i-2,n} \\
		x_{i-1,1} & \cdots & x_{i-1,i-2} & w_{i-1} & x_{i-1,i+1} & \cdots & x_{i-1,n} \\
		x_{i+1,1} & \cdots & x_{i+1,i-2} & w_{i+1} & x_{i+1,i+1} & \cdots & x_{i+1,n} \\
		\vdots & & \vdots & \vdots & & \ddots & \vdots\\
		x_{n,1} & \cdots & x_{n,i-2} & w_{n} & x_{n,i+1} & \cdots & x_{n,n} \\
		\end{array} \right] \label{R2}
\end{align}
where 
\begin{align*}
	w_{k} := \begin{cases}
		 y_{i,1} y_{i,2} \ldots y_{i,k} \cdot x_{k,i-1} \cdot x_{k,i} & k \leq i-1\\
		0 & k > i-1 \quad .
	\end{cases}
\end{align*}
If $i=1$, then $R_{2}(X_{n},Y_{n},i)$ is defined to be the $(n-1) \times (n-1)$ matrix obtained by deleting the $1$-st row and $1$-st column from $X_{n}$. In essence, the $R_{2}$ operator modifies $X_{n}$ by deleting the $i$-th row, merging the $(i-1)$-th column with the $i$-th column (via term-by-term multiplication), and scaling that new column by products of terms from $Y_{n}$.

Lemma~\ref{lemma1324} now directly leads to the following:
\begin{thm}
	For the pattern $\tau = 1 3 2 4$, 
	\begin{align}
		P_{n}(t; X_{n},Y_{n}) = \mathop{\sum} \limits_{i=1}^{n} {x_{i,i}^{n-i} x_{i,i+1}^{n-i-1} \ldots x_{i,n-1}^{1} \cdot P_{n-1}(t; \; R_{2}(X_{n},Y_{n},i), \; R_{1}(Y_{n},i))}. \tag{FE1324} \label{FE1324}
	\end{align}
\end{thm}

Although all the entries in the matrices are changed for consistency and notational convenience, we will continue to disregard the entries below the diagonal in subsequent matrices $X_{k}$ and the entries above the diagonal in subsequent matrices $Y_{k}$. We can apply the same computational tricks shown in \cite{nz:gwilf, naka:gwilf2}. For example, it is not necessary to compute $P_{n}(t; X_{n}, Y_{n})$ completely symbolically and substitute $x_{i,j} = 1$ and $y_{i,j} = 1$ at the end. Instead, we can apply functional equation (\ref{FE1324}) directly to $P_{n}(t; \omat, \omat)$ and subsequent $P_{k}$ terms. We may also use the following lemma from \cite{naka:gwilf2}, which is obvious from the definition of the operator $R_{1}$:
\begin{lemma}
	\label{sqred}
	Let $A$ be a square matrix where every row is identical (i.e., the $i$-th row and the $j$-th row are equal for every $i,j$). Then, $R_{1}(A,i)$ will also be a square matrix with identical rows.
\end{lemma}

By Lemma~\ref{sqred}, repeated applications of the $R_{1}$ operator to the all ones matrix $\omat$ will still result in a matrix with identical rows. It is therefore sufficient to keep track of only a single row. It is also helpful to note that repeated applications of $R_{1}$ to the matrix $\omat$ will result in a matrix whose entries are powers of $t$.

While the lemma does not apply to the $R_{2}$ operator, this still allows us to simplify the polynomial $P_{n}$ and its functional equation by reducing the number of catalytic variables from $n(n+1)$ variables to $n(n+1)/2 + n$ variables.  Let $Q_{n}(t; \; C; \; d_{1}, \ldots, d_{n})$ denote the polynomial $P_{n}(t; C, D)$ where every entry of the $n \times n$ matrices $C$ and $D$ are powers of $t$ and every row in $D$ is $\left[ d_{1}, \ldots, d_{n} \right]$. We get the analogous functional equation:
\begin{gather}
	Q_{n}(t; \; C; \; d_{1}, \ldots, d_{n}) = \notag \\
	\mathop{\sum} \limits_{i=1}^{n} {c_{i,i}^{n-i} c_{i,i+1}^{n-i-1} \ldots c_{i,n-1}^{1} \cdot Q_{n-1}(t; \; R_{2}(C,D,i); \; d_{1}, \ldots, d_{i-1}, t d_{i} d_{i+1}, d_{i+2}, \ldots, d_{n})}. \tag{FE1324c} \label{FE1324c}
\end{gather}
Repeatedly applying this recurrence to $Q_{n}(t; \omat; 1 \left[n \text{ times} \right])$ allows us to compute our desired polynomial since $Q_{n}(t; \omat; 1 \left[n \text{ times} \right]) = P_{n}(t; \omat, \omat) = f_{n}(t)$. Extracting the coefficient of $t^{r}$ from this polynomial gives us $s_{n}(1 3 2 4, r)$.

Additionally, for a fixed $r$, the sequence $s_{n}(1 3 2 4,r)$ can be computed more quickly by discarding higher powers of $t$ (just as in \cite{nz:gwilf, naka:gwilf2}). Although this approach is too memory intensive for larger $r$, for small $r$, this method is still much faster than naive methods that construct the set $\mathcal{S}_{n}(\tau,r)$. The approach has been implemented in the procedure {\tt F1324rN(r,N)} in the accompanying Maple package {\tt F1324}.

For example, the Maple call {\tt F1324rN(0,19);} for the first $19$ terms of $s_{n}(1 3 2 4,0)$ produces the sequence:
\begin{gather*}
	1, 2, 6, 23, 103, 513, 2762, 15793, 94776, 591950, 3824112, 25431452, 173453058,\\
	1209639642, 8604450011, 62300851632, 458374397312, 3421888118907, 25887131596018
\end{gather*}
 and the Maple call {\tt F1324rN(1,17);} for the first $17$ terms of $s_{n}(1 3 2 4,1)$ produces the sequence:
\begin{gather*}
	0, 0, 0, 1, 10, 75, 522, 3579, 24670, 172198, 1219974, 8776255, \\
	64082132, 474605417, 3562460562, 27079243352, 208281537572 .\\
\end{gather*}

\subsection{Specializing to $r=0$}

Unfortunately, the previous algorithm developed for the pattern $1 3 2 4$ is very memory intensive, even for $r=0$. In this subsection, we outline how to extract a simpler recurrence specifically for the pattern avoidance case. 

We will specialize for the $r=0$ case beginning at functional equation (\ref{FE1324c}). Recall that $Q_{n}(t; \; C; \; d_{1}, \ldots, d_{n})$ is the polynomial given by $P_{n}(t; C, D)$ where every entry of the $n \times n$ matrices $C$ and $D$ are powers of $t$ and every row in $D$ is $\left[ d_{1}, \ldots, d_{n} \right]$. We had the functional equation
\begin{gather*}
	Q_{n}(t; \; C; \; d_{1}, \ldots, d_{n}) =\\
	\mathop{\sum} \limits_{i=1}^{n} {c_{i,i}^{n-i} c_{i,i+1}^{n-i-1} \ldots c_{i,n-1}^{1} \cdot Q_{n-1}(t; \; R_{2}(C,D,i) ; \; d_{1}, \ldots, d_{i-1}, t d_{i} d_{i+1}, d_{i+2}, \ldots, d_{n})}
\end{gather*}
and wanted to compute $Q_{n}(t; \; \omat; \; 1 \left[n \text{ times} \right]) = f_{n}(t)$ and extract the coefficient of $t^{r}$, which is exactly $s_{n}(1 3 2 4,r)$.\footnote{Recall that $\omat$ is the $n \times n$ matrix where every entry is $1$.}

Since all the variables $c_{k,l}$ (from matrix $C$) and $d_{k}$ represent powers of $t$, it is sufficient to keep track of powers of $t$ through most of the algorithm. This allows us to consider the analogous function $H_{n}(t; \; U; \; v_{1}, \ldots, v_{n})$, where $U$ is an $n \times n$ matrix of non-negative integers and each $v_{i}$ is a non-negative integer. More precisely, $H_{n}(t; \; U; \; v_{1}, \ldots, v_{n})$ is the polynomial $P_{n}(t; \; C, D)$, where $C$ and $D$ are $n \times n$ matrices, $c_{i,j} = t^{u_{i,j}}$ for every $1 \leq i,j \leq n$, and every row of $D$ is $\left[ t^{v_{1}}, \ldots, t^{v_{n}} \right]$.

In addition, we define the analogous operator $R_{2}^{\prime}$ on an $n \times n$ square matrix $U_{n}$ (of non-negative integers), a length $n$ vector of non-negative integers $\left[ v_{1}, \ldots, v_{n} \right]$, and $1 < i \leq n$: 
\begin{align}
	R_{2}^{\prime}(U_{n},\left[ v_{1}, \ldots, v_{n} \right],i) := \left[ \begin{array}{ccccccc} 
		u_{1,1} & \cdots & u_{1,i-2} & w_{1}^{\prime} & u_{1,i+1} & \cdots & u_{1,n} \\
		\vdots & \ddots & & \vdots & & & \vdots \\
		u_{i-2,1} & \cdots & u_{i-2,i-2} & w_{i-2}^{\prime} & u_{i-2,i+1} & \cdots & u_{i-2,n} \\
		u_{i-1,1} & \cdots & u_{i-1,i-2} & w_{i-1}^{\prime} & u_{i-1,i+1} & \cdots & u_{i-1,n} \\
		u_{i+1,1} & \cdots & u_{i+1,i-2} & w_{i+1}^{\prime} & u_{i+1,i+1} & \cdots & u_{i+1,n} \\
		\vdots & & \vdots & \vdots & & \ddots & \vdots\\
		u_{n,1} & \cdots & u_{n,i-2} & w_{n}^{\prime} & u_{n,i+1} & \cdots & u_{n,n} \\
		\end{array} \right] \label{R2e}
\end{align}
where 
\begin{align}
	\label{R2ew}
	w_{k}^{\prime} := \begin{cases}
		 (v_{1} + v_{2} + \ldots + v_{k}) + u_{k,i-1} + u_{k,i} & k \leq i-1\\
		0 & k > i-1 \quad .
	\end{cases}
\end{align}
If $i=1$, then $R_{2}^{\prime}(U_{n},\left[ v_{1}, \ldots, v_{n} \right],i)$ is defined to be the $(n-1) \times (n-1)$ matrix obtained by deleting the $1$-st row and $1$-st column from $U_{n}$. In essence, the $R_{2}^{\prime}$ operator modifies $U_{n}$ by deleting the $i$-th row, merging the $(i-1)$-th column with the $i$-th column (via term-by-term addition), and adding partial sums of $\left[ v_{1}, \ldots, v_{n} \right]$ into the new column.

We now have the functional equation (analogous to Eq. (\ref{FE1324c})):
\begin{gather*}
	H_{n}(t; \; U; \; v_{1}, \ldots, v_{n}) =\\
	\mathop{\sum} \limits_{i=1}^{n} {t^{e_{i}} \cdot H_{n-1}(t; \; R_{2}^{\prime}(U, \left[ v_{1}, \ldots, v_{n} \right] ,i) ; \; v_{1}, \ldots, v_{i-1}, (v_{i} + v_{i+1} + 1), v_{i+2}, \ldots, v_{n})} \tag{FE1324e} \label{FE1324e}
\end{gather*}
where $e_{i} = (n-i) u_{i,i} + (n-i-1) u_{i,i+1} + \ldots + (1) u_{i,n-1}$. Observe that $H_{n}(t; \; \zmat; \; 0 \left[n \text{ times} \right])$ is now our desired polynomial $f_{n}(t)$.\footnote{We denote the $n \times n$ matrix consisting of all zeros by $\zmat$.}

Since we are specifically considering the $r=0$ case, we can make additional observations and simplifications. First, we are only interested in the constant term of $f_{n}(t)$. As in \cite{nz:gwilf, naka:gwilf2}, we only need to keep track of polynomials of the form $a_{0} + a_{1} t$ in intermediate computations, where $a_{1}$ represents permutations with \emph{at least one} occurrence of the pattern we are tracking. Because of this, we may consider all matrices/vectors used in $H_{n}$ to be $0$-$1$ matrices/vectors. After every addition (for example, in the $w_{k}^{\prime}$ term in $R_{2}^{\prime}$), we can take the minimum of the resulting sum and $1$.

Next, observe that $v_{1}, \ldots, v_{n}$ only appear in the $R_{2}^{\prime}$ operator, and in particular, in the partial sums for $w_{k}^{\prime}$ in Eq.~\ref{R2ew}. Suppose that some of the $v_{1}, \ldots, v_{n}$ are equal to $1$, and let $j$ be the smallest number such that $v_{j} = 1$. Then, 
\begin{align*}
	H_{n}(t; \; U; \; v_{1}, \ldots, v_{n}) |_{t=0} = H_{n}(t; \; U; \; 0 \left[j-1 \text{ times} \right],  1 \left[n-j+1 \text{ times} \right] ) |_{t=0}  .
\end{align*}
In particular, the variables $v_{1}, \ldots, v_{n}$ are unnecessary, and it is sufficient to keep track of how many $0$'s there are. We can then consider this slightly simpler function
\begin{align*}
	\widetilde{H}_{n}(t; \; U; \; k) := H_{n}(t; \; U; \; 0 \left[k \text{ times} \right],  1 \left[n-k \text{ times} \right] ) 
\end{align*}
where $0 \leq k \leq n$.

Finally, observe that whenever $e_{i} > 0$ in (\ref{FE1324e}), we can discard the entire term since we are only interested in the constant term of the final polynomial. Observe that $e_{i} > 0$ if and only if $u_{i,j} > 0$ for some $j \geq i$. This observation (combined with how the $R_{2}^{\prime}$ operator ``modifies'' the matrix $U_{n}$) implies that we only need to keep track of the left-most $1$ within each row of $U_{n}$. If there are multiple $1$'s on a row, the left-most $1$ is sufficient to force $e_{i} > 0$ as long as it is not in the $n$-th column. Therefore, we can consider a function of the form 
\begin{align*}
	H_{n}^{0}(t; \; b_{1}, \ldots, b_{n}; \; k)  := \widetilde{H}_{n}(t; \; B_{n}; \; k) = H_{n}(t; \; B_{n}; \; 0 \left[k \text{ times} \right],  1 \left[n-k \text{ times} \right] ) 
\end{align*}
where $0 \leq k \leq n$ and $1 \leq b_{j} \leq n+1$ for each $j$ and $B_{n}$ is the $n \times n$ matrix where the $j$-th row is $\left[ \; 0 \left[n \text{ times} \right] \; \right]$ if $b_{j} = n+1$ and otherwise is $\left[ \; 0 \left[b_{j}-1 \text{ times} \right], \; 1 \left[n-b_{j}+1 \text{ times} \right] \; \right]$.

This approach is implemented in the Maple package {\tt F1324} in the procedure {\tt AV1324(n)}. An improved implementation in C++ is provided in the program {\tt av1324.cpp} and is discussed in greater detail in the next section.\\

\section{Computational details and results}\label{compres}

\subsection{Algorithmic details for enumerating $1324$-avoiders}

For notational convenience, let $a_{n} \equiv s_{n}(1 3 2 4,0)$ denote the number of 1324-avoiding permutations of length $n$. The values $a_{n}$ up to $n = 25$ were previously computed in \cite{aerwz} and subsequently listed in \href{http://oeis.org/A061552}{A061552} of OEIS (\cite{OEIS}). To extend this list, we have written a C++ implementation of the algorithm described in the previous section. The implementation is in the program {\tt av1324.cpp}, which is available from the authors' websites.

The main part of the program is a recursive function $G(n,k,b) \equiv H_{n}^{0}(t; \; b_{1}, \ldots, b_{n}; \; k)$ which takes as input an integer $n \leq 1$, an integer $0 \leq k \leq n$, and a vector of integers $b = [\tilde b_{0}, \ldots \tilde b_{n-1}]$ satisfying $0 \leq \tilde b_{j} \leq n$ (these correspond to the $b_{j}$ in $H_{n}^{0}$ as $\tilde b_{j} = b_{j+1} - 1$, being zero-aligned since this is more natural in C++). We represent $b$ as an array of bytes, zero-padded to a fixed maximum length of 32 (sufficient to compute $a_{n}$ up to $n = 32$). The output of $G$ is an integer, which we represent as a 128-bit unsigned integer (sufficient to compute $a_{n}$ up to at least $n = 34$, since $a_{34} < 34! < 2^{128}$ and all recursive calls to $G$ must produce values that are no larger than the output).

We use full memorization to reduce the number of recursive calls that have to be made. The {\tt std::map} type in the C++ standard library is used to associate vectors $b$ with output values $G(n,k,b)$, using one such map for each pair $n, k$. The {\tt std::map} type implements a self-balancing binary tree with $O(\log N)$ insertion and lookup time where $N$ is the size of the cache. The keys $b$ are compared lexicographically by casting to 64-bit integers and processing eight bytes at a time.

We compiled the program with GCC 4.3.4 and ran it on the MACH computer at the Johannes Kepler University of Linz, using a 2.66 GHz Intel Xeon E7-8837 CPU with 1024 GiB of memory allocated to the process. The memory limit allowed computing $a_n$ up to $n = 31$. The $6$ new terms are:
\begin{align*}
	a_{26} &= 49339914891701589053\\
	a_{27} &= 402890652358573525928\\
	a_{28} &= 3313004165660965754922\\
	a_{29} &= 27424185239545986820514\\
	a_{30} &= 228437994561962363104048\\
	a_{31} &= 1914189093351633702834757
\end{align*}
We computed all $a_n$ consecutively in one run to benefit from already cached function values (computing a single $a_n$ in isolation would not give any significant memory savings). With $a_{1}, \ldots, a_{30}$ already computed, the evaluation of $a_{31}$ took 33 hours and used 920 GiB of memory, and the computation as a whole took 50 hours.

Detailed results from the computation are presented in Table~\ref{tab:compresults} (in the Appendix). Besides the running time and total memory usage, we record the number of cache hits (calls to $H$ replaced by cache lookups) and the number of cache misses (calls to $G$ that require evaluation). Up to $n = 15$, we also show the number of calls to $G$ used by the algorithm with caching disabled, measured in a separate run of the program.

All quantities appear to grow slightly faster than exponentially. Over the measured range, the number of function calls (as well as the size of the cache in memory) is roughly proportional to $2.2^{n}$. If the memory overhead of the implementation were reduced by half (or if we had a computer with twice as much memory), we could thus compute roughly one more entry. With caching disabled, the number of recursive function calls grows much faster, making this version impractical (for $n \leq 15$, the number of calls is roughly $2.4 a_{n}$).

At present, we do not know if the algorithm can be modified to significantly reduce the memory consumption required by full memorization, without significantly increasing the running time. Such a modification might allow computing several more entries in the sequence $\{a_{n}\}$, particularly if the job could be parallelized.

\subsection{Observations on the asymptotics}

Since the enumeration problem is solved for the patterns $1 2 3 4$ and $1 3 4 2$, it is not hard to derive asymptotic information on the sequences enumerating their respective pattern avoiders. For the pattern $1 2 3 4$, we have
\begin{align*}
	s_{n}(1 2 3 4, 0) \sim 9^{n} n^{-4}
\end{align*}
and for the pattern $1 3 4 2$, we have
\begin{align*}
	s_{n}(1 3 4 2, 0) \sim 8^{n} n^{-5/2}.
\end{align*}
The convergence occurs fairly quickly, even when observing the first $31$ terms. On the other hand, it is not even known if the asymptotics for $s_{n}(1 3 2 4, 0)$ look like $\mu^{n} n^{\theta}$ (for constants $\mu$ and $\theta$). Shalosh B.~Ekhad was kind enough to compute some numerical data for the authors.

Using the first $29$ terms, the first $30$ terms, and the first $31$ terms, the approximate value for $\theta$ and $\mu$ for the patterns $1234$ and $1342$ are:
\begin{table}[h]
	\parbox{.45\linewidth}{
	\centering
	\begin{tabular}{|c|c|c|}
		\hline
		n & $\theta$ & $\mu$\\
		\hline
		29 & -3.990065278 & 8.978066441\\
		\hline
		30 & -3.990767318 & 8.979528508\\
		\hline
		31 &  -3.991374852 & 8.980845382\\
		\hline
	\end{tabular}
	\caption{Approximate values for $\theta$ and $\mu$ using the first $n=29, 30, 31$ terms of $s_{n}(1 2 3 4,0)$.}
	}
	\hfill
	\parbox{.45\linewidth}{
	\centering
		\begin{tabular}{|c|c|c|}
		\hline
		n & $\theta$ & $\mu$\\
		\hline
		29 & -2.507234370 & 7.987629199\\
		\hline
		30 & -2.506672206 & 7.988446482\\
		\hline
		31 & -2.506140202 & 7.989181549\\
		\hline
	\end{tabular}
	\caption{Approximate values for $\theta$ and $\mu$ using the first $n=29, 30, 31$ terms of $s_{n}(1 3 4 2,0)$.}
	}
\end{table}

However, when the same guessing is done for the pattern $1 3 2 4$, the convergence is much slower. The values are:
\begin{table}[h]
	\centering
	\begin{tabular}{|c|c|c|}
		\hline
		n & $\theta$ & $\mu$\\
		\hline
		29 & -8.365614110 & 10.40595402\\
		\hline
		30 & -8.506078382 & 10.42830233\\
		\hline
		31 & -8.643316748 & 10.44936383\\
		\hline
	\end{tabular}
	\caption{``Approximate'' values for $\theta$ and $\mu$ using the first $n=29, 30, 31$ terms of $s_{n}(1 3 2 4,0)$.}
\end{table}

It should be re-emphasized that it is not known whether this sequence fits the asymptotic form of $\mu^{n} n^{\theta}$. The empirical asymptotics suggests that either the convergence is much slower than the other length $4$ patterns or that the sequence does not have that asymptotic form to begin with (unlike the other patterns). In addition, the $\mu$ values would suggest that the Stanley-Wilf limit $L(1 3 2 4)$ is at least $10.45$. More detailed numerical data on the asymptotics (from Shalosh B.~Ekhad) can be found on the authors' websites.\\

\section{Extending to inversions}\label{permstats}

In this section, we show how the previous functional equations can be adapted to refine the values by the number of inversions. The number of inversions in a permutation is one of the most commonly studied permutation statistic and in essence, quantifies how ``unsorted'' a permutation is. Given a permutation $\pi = \pi_{1} \ldots \pi_{n}$, the \emph{inversion number} of $\pi$, denoted by $\inv(\pi)$, is the number of pairs $(i,j)$ such that $1 \leq i < j \leq n$ and $\pi_{i} > \pi_{j}$. An equivalent definition is that $\inv(\pi) = N_{2 1}(\pi)$, the number of $2 1$ patterns in $\pi$.

We again consider the pattern $1 3 2 4$. For each $n$, we define the bivariate polynomial
\begin{align*}
	g_{n}(t,q) := \mathop{\sum} \limits_{\pi \in \mathcal{S}_{n}} {q^{\inv(\pi)} t^{N_{1 3 2 4}(\pi)}}.
\end{align*}
Observe that $g_{n}(t,1)$ is exactly $f_{n}(t)$ from Section~\ref{pat1324}.

For the pattern $\tau = 1 3 2 4$, the polynomial $P_{n}$ can now be ``generalized'' as
\begin{align*}
	P_{n}(t,q; \; X_{n}, Y_{n}) := \mathop{\sum} \limits_{\pi \in \mathcal{S}_{n}} {q^{\inv(\pi)} \weight(\pi)}.
\end{align*}
Note that $P_{n}(t,1; \; X_{n}, Y_{n})$ is exactly $P_{n}(t; \; X_{n}, Y_{n})$ from Section~\ref{pat1324}.

We now make an important observation. Given a permutation $\pi = \pi_{1} \ldots \pi_{n}$, suppose that $\pi_{1} = i$. Then, $\inv(\pi)$ is equal to the number of inversions in $\pi_{2} \ldots \pi_{n}$ \textbf{plus} the number of terms in $\pi_{2} \ldots \pi_{n}$ that are less than $i$ (which is exactly $i-1$). For any previously define functional equation, it is enough to insert a factor of $q^{i-1}$ into the summation.

We can now quickly derive the modified functional equations for the pattern. The functional equation (\ref{FE1324}) now becomes:
\begin{corollary}
	For the pattern $\tau = 1 3 2 4$, 
	\begin{align}
		P_{n}(t,q; X_{n},Y_{n}) = \mathop{\sum} \limits_{i=1}^{n} {q^{i-1} x_{i,i}^{n-i} x_{i,i+1}^{n-i-1} \ldots x_{i,n-1}^{1} \cdot P_{n-1}(t,q; \; R_{2}(X_{n},Y_{n},i), \; R_{1}(Y_{n},i))}. \tag{qFE1324} \label{qFE1324}
	\end{align}
\end{corollary}
Similarly, the functional equation (\ref{FE1324c}) for $Q_{n}(t; \; C; \; d_{1}, \ldots, d_{n})$ now becomes the analogous:
\begin{corollary}
	For the pattern $\tau = 1 3 2 4$, 
	\begin{gather}
		Q_{n}(t,q ; \; C; \; d_{1}, \ldots, d_{n}) = \notag \\
		\mathop{\sum} \limits_{i=1}^{n} {q^{i-1} c_{i,i}^{n-i} c_{i,i+1}^{n-i-1} \ldots c_{i,n-1}^{1} \cdot Q_{n-1}(t,q; \; R_{2}(C,D,i); \; d_{1}, \ldots, d_{i-1}, t d_{i} d_{i+1}, d_{i+2}, \ldots, d_{n})}. \tag{qFE1324c} \label{qFE1324c}
	\end{gather}
\end{corollary}
\noindent The functional equation (\ref{FE1324e}) for $H_{n}(t; \; U; \; v_{1}, \ldots, v_{n})$, which corresponds specifically to the pattern-avoidance case, also becomes the analogous:
\begin{corollary}
	For the pattern $\tau = 1 3 2 4$, 
	\begin{gather*}
		H_{n}(t,q; \; U; \; v_{1}, \ldots, v_{n}) =\\
		\mathop{\sum} \limits_{i=1}^{n} {q^{i-1} t^{e_{i}} \cdot H_{n-1}(t,q; \; R_{2}^{\prime}(U, \left[ v_{1}, \ldots, v_{n} \right] ,i) ; \; v_{1}, \ldots, v_{i-1}, (v_{i} + v_{i+1} + 1), v_{i+2}, \ldots, v_{n})} \tag{qFE1324e} \label{qFE1324e}
	\end{gather*}
\end{corollary}

The enumeration algorithm derived from the functional equation (\ref{qFE1324e}) has been implemented in the procedure {\tt qAv1324r(n,r,q)} in the Maple package {\tt F1324}. It is also worth noting that the same extension can be done for tracking non-inversions by inserting $q^{n-i}$ into the summations (instead of $q^{i-1}$). This has also been implemented in the Maple package {\tt F1324} in the procedure {\tt pAv1324r(n,r,p)}.

In \cite{cjs:ub1324}, Claesson, Jel{\'{\i}}nek, and Steingr{\'{\i}}msson consider refining the number of $1324$-avoiding permutations by the inversion number to prove a new upper bound on the Stanley-Wilf limit. They conjectured that for each fixed $k$, the number of $1324$-avoiders with exactly $k$ inversions is non-decreasing in $n$. If this conjecture is proven true, their result would improve the upper bound of the growth rate. Using our algorithm, we are able to empirically confirm this conjecture for $n,k \leq 23$, and the explicit values can be found on the authors' websites. We suspect that a more careful analysis of the functional equations and the algorithm may lead to new insights regarding this conjecture.\\

\section{Conclusion}\label{concl}

In this paper, we extended the functional equations approach of \cite{nz:gwilf, naka:gwilf2} to derive functional equations as well as an enumeration algorithm for computing $s_{n}(1 3 2 4,r)$. We were able to specialize this approach to $r=0$ to develop a new enumeration algorithm for computing the number of $1 3 2 4$-avoiding permutations. This new approach was used to compute $6$ new terms for the sequence as well as make some empirical observations on the asymptotics. The functional equations were also extended to refine $s_{n}(1 3 2 4, r)$ by the number of inversions.

While some of the key contributions of this paper are algorithms, it is important to note that all the intermediate steps were rigorous functional equations. In particular, the specialized functional equations for enumerating $1 3 2 4$-avoiders can be viewed as recursively defined functions on $0$-$1$ matrices. We suspect that a more careful analysis of these functions may yield insight into the sequence enumerating the $1 3 2 4$-avoiders and perhaps the Stanley-Wilf limit $L(1 3 2 4)$. In addition, the functional equations for tracking inversions were used to confirm the conjecture by Claesson, Jel{\'{\i}}nek, and Steingr{\'{\i}}msson for up to $n,k \leq 23$. We suspect that a more careful analysis of these functional equations may also provide insight toward resolving their conjecture, which would then lower the upper bound on $L(1 3 2 4)$ to $e^{\pi \sqrt{2/3}} \approx 13.002$.\\

\noindent \textbf{Acknowledgments}: We would like to thank Doron Zeilberger and Manuel Kauers for their very helpful comments and suggestions. We would also like to thank Shalosh B.~Ekhad for computing numerical data on the asymptotics for us.

\bibliography{FE1324b}{}
\bibliographystyle{plain}

\clearpage

\appendix
\section*{Appendix}

\begin{table}[h]
\centering
\begin{small}
\begin{tabular}{ l l l l c c}
\hline
$n$ & $a_n = s_n(1 3 2 4,0)$ & Cache hits & Cache misses & \multicolumn{2}{l}{Calls, cache disabled} \\
\hline
$1$  &  $1$  &  $0$  &  $1$  &  \multicolumn{2}{l}{1}    \\
$2$  &  $2$  &  $0$  &  $4$  &  \multicolumn{2}{l}{4}    \\
$3$  &  $6$  &  $4$  &  $10$  &  \multicolumn{2}{l}{14}    \\
$4$  &  $23$  &  $17$  &  $21$  &  \multicolumn{2}{l}{54}    \\
$5$  &  $103$  &  $49$  &  $41$  &  \multicolumn{2}{l}{239}    \\
$6$  &  $513$  &  $121$  &  $79$  &  \multicolumn{2}{l}{1187}    \\
$7$  &  $2762$  &  $280$  &  $153$  &  \multicolumn{2}{l}{6417}    \\
$8$  &  $15793$  &  $628$  &  $300$  &  \multicolumn{2}{l}{36936}    \\
$9$  &  $94776$  &  $1386$  &  $595$  &  \multicolumn{2}{l}{223190}    \\
$10$  &  $591950$  &  $3032$  &  $1194$  &  \multicolumn{2}{l}{1402845}    \\
$11$  &  $3824112$  &  $6607$  &  $2422$  &  \multicolumn{2}{l}{9113389}    \\
$12$  &  $25431452$  &  $14383$  &  $4963$  &  \multicolumn{2}{l}{60903142}    \\
$13$  &  $173453058$  &  $31328$  &  $10260$  &  \multicolumn{2}{l}{417167046}    \\
$14$  &  $1209639642$  &  $68314$  &  $21375$  &  \multicolumn{2}{l}{2920322177}    \\
$15$  &  $8604450011$  &  $149166$  &  $44828$  &   \multicolumn{2}{l}{20843563430}   \\
$16$  &  $62300851632$  &  $326163$  &  $94562$  &   & \\ \cline{5-6}
$17$  &  $458374397312$  &  $714178$  &  $200491$  &  Time, ms  &  Mem, KiB  \\ \cline{5-6}
$18$  &  $3421888118907$  &  $1565935$  &  $427006$  &  $3.4\cdot 10^{2}$  &  $5.2\cdot 10^{4}$  \\
$19$  &  $25887131596018$  &  $3438097$  &  $913101$  &  $9.1\cdot 10^{2}$  &  $9.7\cdot 10^{4}$  \\
$20$  &  $198244731603623$  &  $7558183$  &  $1959618$  &  $2.5\cdot 10^{3}$  &  $2.0\cdot 10^{5}$  \\
$21$  &  $1535346218316422$  &  $16636000$  &  $4219286$  &  $6.4\cdot 10^{3}$  &  $4.1\cdot 10^{5}$  \\
$22$  &  $12015325816028313$  &  $36659838$  &  $9111542$  &  $1.6\cdot 10^{4}$  &  $8.7\cdot 10^{5}$  \\
$23$  &  $94944352095728825$  &  $80876277$  &  $19729578$  &  $4.1\cdot 10^{4}$  &  $1.9\cdot 10^{6}$  \\
$24$  &  $757046484552152932$  &  $178616038$  &  $42827166$  &  $1.0\cdot 10^{5}$  &  $4.0\cdot 10^{6}$  \\
$25$  &  $6087537591051072864$  &  $394883523$  &  $93177487$  &  $2.6\cdot 10^{5}$  &  $8.7\cdot 10^{6}$  \\
$26$  &  $49339914891701589053$  &  $873872819$  &  $203150306$  &  $6.5\cdot 10^{5}$  &  $1.9\cdot 10^{7}$  \\
$27$  &  $402890652358573525928$  &  $1935710217$  &  $443784326$  &  $1.6\cdot 10^{6}$  &  $4.2\cdot 10^{7}$  \\
$28$  &  $3313004165660965754922$  &  $4291690537$  &  $971213858$  &  $4.3\cdot 10^{6}$  &  $9.1\cdot 10^{7}$  \\
$29$  &  $27424185239545986820514$  &  $9523492671$  &  $2129084186$  &  $1.4\cdot 10^{7}$  &  $2.0\cdot 10^{8}$  \\
$30$  &  $228437994561962363104048$  &  $21150884205$  &  $4674743970$  &  $4.2\cdot 10^{7}$  &  $4.4\cdot 10^{8}$  \\
$31$  &  $1914189093351633702834757$  &  $47012202538$  &  $10279369333$  &  $1.2\cdot 10^{8}$  &  $9.6\cdot 10^{8}$  \\
\end{tabular}
\caption{Results of the computation of the number of 1324-avoiding permutations
of length $n$ using our C++ implementation.
The table includes the number of calls to the function $G$
that result in a cache hit and the number of calls that result in a cache miss.
For small $n$, the number of evaluations $G$ that would be required
without memorization are presented in the rightmost column. For large $n$,
the elapsed CPU time and memory usage of the memorized version are shown.}
\label{tab:compresults}
\end{small}
\end{table}

\end{document}